\documentclass{article}
\usepackage[english]{babel}
\usepackage{amsthm}
\usepackage[utf8]{inputenc}
\usepackage{hyperref}
\usepackage{authblk}
\usepackage[all]{xy}
\usepackage{amssymb}
\usepackage{amsmath}
\usepackage{color}
\usepackage{setspace}
\usepackage{lipsum}

\font\cyr=wncyr10 scaled \magstep1%
\def\Sh{\text{\cyr Sh}}

 \newcommand{\xrightarrowdbl}[2][]{%
   \xrightarrow[#1]{#2}\mathrel{\mkern-14mu}\rightarrow
}
\newtheorem{theorem}{Theorem}[section]

\newtheorem{prop}[theorem]{Proposition}
\newtheorem{lem}[theorem]{Lemma}
\newtheorem{cor}[theorem]{Corollary}

\theoremstyle{definition}

\newtheorem{remark}[theorem]{Remark}
\newtheorem{example}[theorem]{Example}

\numberwithin{equation}{section}

\newtheorem*{maintheorem*}{Main Theorem}
\theoremstyle{definition}
\newtheorem{definition}{Definition}

\newcommand{\CO}{\mathcal{O}}
\newcommand{\hCO}{\hat{\mathcal{O}}}

\newcommand{\SO}{\un{\mathrm{SO}}_q}

\newcommand{\un}{\underline}
\newcommand{\B}{\mathcal{B}}

\newcommand{\wt}{\widetilde}

\newcommand{\mk}{\medskip}
\renewcommand{\sectionmark}[1]{}
\renewcommand{\Im}{\operatorname{Im}}

\newcommand{\la}{\langle}
\newcommand{\ra}{\rangle}
\newcommand{\N}{\un{\mathrm{N}}}

\newcommand{\diag}{\text{diag}}

\newcommand{\af}{\text{af}}

\newcommand{\iy}{\infty}

\newcommand{\bk}{\bigskip}
\newcommand{\fc}{\frac}

\newcommand{\g}{\gamma}

\newcommand{\s}{\sigma}

\newcommand{\Pic}{\text{Pic}}

\newcommand{\dl}{\delta}
\newcommand{\Dl}{\Delta}
\newcommand{\lm}{\lambda}

\newcommand{\om}{\omega}
\newcommand{\Om}{\Omega}
\newcommand{\ov}{\overline}

\newcommand{\vp}{\varphi}

\newcommand{\BG}{\mathbb{G}}
\newcommand{\BF}{\mathbb{F}}
\newcommand{\C}{\mathcal{C}}

\newcommand{\Q}{\mathbb{Q}}
\newcommand{\Z}{\mathbb{Z}}

\renewcommand{\a}{\alpha}

\newcommand{\cl}{\mathfrak{cl}}

\newcommand{\op}{\text{op}}

\newcommand{\et}{\text{\'et}}

\newcommand{\p}{\varphi}

\newcommand{\A}{\mathbb{A}}
\newcommand{\mA}{\mathcal{A}}

\newcommand{\Nr}{\text{Nr}}

\newcommand{\fp}{\mathfrak{p}}

\newcommand{\fP}{\mathfrak{P}}

\newcommand{\Iso}{\text{Iso}}

\newcommand{\Oq}{\un{\mathrm{O}}_q}

\newcommand{\Sp}{\text{Spec} \,}

\newcommand{\disc}{\text{disc}}

\begin{document}

\title{The Geometric Gauss-Dedekind} 
\date{\vspace{-5ex}}
\author{Rony A.~Bitan}



\maketitle

\sloppy
\begin{abstract}
\noindent Gauss and Dedekind have shown a bijection 
between the set of 
$\mathrm{SL}_2(\mathbb{Z})$-equivalence classes 
of positive binary quadratic $\Z$-forms 
of the discriminant of an imaginary quadratic 
field and the class group of its ring of integers. 
Using \'etale cohomology we show an analogue 
of this correspondence in the positive characteristic. 
This leads to the description of the 
set of genera and to another result 
analogous to Gauss' one by which 
any form composed with it belongs 
to the principal genus.  
\end{abstract}
\fussy

\bigskip

\section{Introduction} \label{intro}
Let $R$ be a unital commutative domain, 
$M$ an $R$-module of rank $2$ and $N$ 
an invertible $R$-module. 
A map $q:M \to N$ is called a \emph{binary quadratic map}  
if $q(rx)= r^2 q(x)$ for all $x \in M, r \in R$,  
and the induced symmetric map  
$$ B_q : M \times M \to N: \ (x,y) \mapsto q(x+y)-q(x)-q(y)$$ 
is $R$-bilinear (\cite[\S 6]{Kne}). 
In particular if $N=R$, this $q$ 
is the familiar binary \emph{quadratic form} 
(\cite[\S 5.3.5]{Knus}). 
The quadratic module $(M,q)$ is called 
\emph{primitive} if $Rq(M) = N$.  
In that case, fixing an $R$-basis $\{e_1,e_2\}$ 
of $M$, there exist co-prime elements 
$a,b,c \in R$ such that: 
$$\tilde{q}(X,Y) := q(Xe_1+Ye_2) = aX^2 + bXY + cY^2. $$  
For brevity we denote $q=(a,b,c)$. 
Throughout, all quadratic maps will be assumed primitive. 
  
\mk 

Let $A$ be a quadratic $R$-algebra. 
It admits a standard involution $\s$,   
thus a norm map $n_A:A \to R: x \mapsto x\s(x)$.   
Given an invertible $A$-module $M$  
equipped with a quadratic map $q:M \to N$,   
we say that $(M,q)$ is \emph{of type~$A$}, 
if $q$ is compatible with $n_A$, 
namely, if it satisfies $q(xa)=q(x)n_A(a)$ 
for all $x \in M, a \in A$ (cf. \cite[III, \S 7.2]{Knus}). 
M.~Kneser showed in (\cite[Prop. 2]{Kne}), that 
given an invertible $A$-module $M$, there exists 
a unique pair $(N,q)$ up to isomorphism, 
such that $(M,q)$ is primitive of type $A$. 

\mk 

A change of variables by $H \in \mathrm{GL}(M)$,  
such that the quadratic map 
\mbox{$q'=q \circ H$} is still defined over $R$, 
is called an \emph{isometry} of $q$. 
It is \emph{proper} if $\det(H)=1$ and \emph{improper} otherwise. 
If $H$ is proper and defined over an $R$-algebra $S$, 
then $q'$ is said to be \emph{equivalent modulo $\mathrm{SL}_2(S)$} to $q$. 
The \emph{discriminant} of $q$, up to $R$-equivalence, 
is defined as the coset (\cite[\S 2.1]{Con2})
$$ 
\disc(q) := \det(B_q) \cdot (R^\times)^2 \in R/(R^\times)^2,
$$ 
where $\det(B_q)$ stands for $\det(B_q(e_i,e_j))$.  

\begin{definition}  \label{cl1} 
Given $\Delta \in R/(R^\times)^2$ we set: 
\begin{align*}
\cl_0(\Dl) &:= \{\text{primitive quadratic forms} \ q:\disc(q) = \Dl\} / \mathrm{SL}_2(R) \\ \nonumber \subseteq 
\cl_1(\Dl) &:= \{\text{primitive quadratic maps} \ q:\disc(q) = \Dl\} / \mathrm{SL}_2(R).    
\end{align*}   
If $\Pic(R)=1$ then $\cl_0(\Dl) = \cl_1(\Dl)$. 
A coarser classification, up to also improper isometries, is:  
$$ 
\cl(\Dl) := \{ q:\disc(q) = \Dl\} / \mathrm{GL}_2(R). 
$$   
Over a domain of positive characteristic there 
is no notion of a definite quadratic form (positive or negative). 
We may generalize this concept by defining 
\emph{definite forms} as the representatives in 
$$ \cl_1'(\Dl) := \cl_1(\Dl) / [q] \sim [\lm q]\ : \  \lm \in R^\times - (R^\times)^2.  $$
Over $R=\Z$ for which $(R^\times)^2=1$,   
a discriminant is an integer $d$ and 
$$ 
[q] \sim [\lm q]\ : \  \lm \in R^\times - (R^\times)^2 
$$
gives $[q] \sim [-q]$, which means -- when $d<0$ -- 
that any positive definite form is identified with its negative, thus $\cl_1'(d)$ 
classifies the positive ones. 
\end{definition}

Gauss, in his famous 
{\emph{Disquisitiones Arithmeticae}}~\cite{Gau}, 
defined the composition law of two binary 
quadratic $R$-forms. 
Later, R.\ Dedekind identified the obtained 
group structure with the one of an abelian group; 
given an imaginary quadratic field $K=\Q(\sqrt{d})$, 
$d < 0$ is squarefree,   
with ring of integers $\CO_K$ and 
discriminant $\Dl_K$ there is a bijection 
of pointed sets: 
$$ 
\cl_1'(\Dl_K) \xrightarrow{\sim} \mathrm{Pic}
(\CO_K): \ \left[ (a,b,c) \right] \mapsto  
\left [ \left\la a,\fc{b - \sqrt{\Dl_K}}{2} \right \ra \right],  
$$  
where 
$\Pic(\CO_K)$ is the Picard group of $\CO_K$ 
(\cite[Thm.~58]{FT}). 

\mk 

M.\ Kneser showed in \cite[p.412]{Kne}, 1982, that  quadratic modules of type~$A$, 
where $A$ is a quadratic $R$-algebra, are classified by $\Pic(A)$.   
In 1991, M.\ A.\ Knus in \cite[IV \S 5]{Knus},    
used \'etale cohomology in the nondegenerate case 
to classify regular quadratic $R$-forms with trivial Arf invariant. 
M.\ M.\ Wood has proved in \cite{Wood}, 2011,  
using actions on the symmetric space of a quadratic module, 
a similar correspondence over $R$, but without restricting to positive definite forms 
(which are not defined in the positive characteristic). 

\mk 

In this paper, assuming $R$ is 
a domain in which $2$ is invertible, 
due to a recent result of \cite{APS} regarding the 
smoothness of the special orthogonal group 
$\SO$ of a quadratic form $q$ with squarefree 
discriminant, thus being primitive (though maybe degenerate), 
we generalize the cohomological approach to such 
quadratic forms and maps. 

\mk 

In Section \ref{torsors} we form a bijection of pointed sets, 
having a structure of an abelian group: $\cl_0(\disc(q)) \cong H^1_\et(R,\SO)$.   
The above new definition of definite forms  
leads us in Section~\ref{Gauss section} to formulate the following correspondence, 
which is a general constructive analogue of the one of Dedekind, 
in the geometric case (Theorem~\ref{Gauss} below): 

\begin{theorem}  
Given a finite field $\BF$ of odd characteristic, 
let $k/\BF(x)$ be an imaginary extension and $\CO$ the domain of $k$-elements regular
everywhere away of the prime $\iy_k$ lying over $\iy=\la 1/x \ra$.    
Let $K/k$ be a geometric imaginary quadratic extension and $\CO_K$ 
the domain of $K$-elements regular everywhere away of the prime $\iy_K|\iy_k$, 
with discriminant $\Dl_K = -\a (\BF^\times)^2$, $\a \in \CO$. 
There is an isomorphism of abelian groups:    
$$ \wt{i}_*:\cl_1'(\Dl_K) \xrightarrow{\sim} \mathrm{Pic}(\CO_K): 
\ \left[(a,b,c) \right] 
\mapsto \left[ \big \la a,b/2 + \sqrt{\a} \big \ra \right]. $$ 
\end{theorem}

\mk 

In Section \ref{improper classification} we show that the above coarser 
classification $\cl(\Dl)$ (which is not necessarily a group),   
is bijective to: $\cl_1(\Dl)/([q] \sim [q^\op])$,  
where for $q=(a,b,c)$, $q^\op = (a,-b,c)$. 
Similarly: 
$\cl'(\Dl_K) = \cl_1'(\Dl_K)/([q] \sim [q^\op])$. 

\mk 

For any prime $\fp$ of $k$ let $\hCO_\fp$ 
be the completion of $\CO$ with respect to 
the discrete valuation induced by $\fp$.  
The \emph{principal genus} $\text{Cl}_\iy(q)$ of $q$ 
is the set of classes of $\CO$-forms that are properly 
$K$- and $\hCO_\fp$-isomorphic to $q$ for any $\fp$. 
In Section \ref{class set}, after dividing the 
classes in $\cl_1(\disc(q))$ into genera, 
we prove in Corollary \ref{exponent 2} 
another result analogous to a one of Gauss in characteristic $0$: 
denote by $\star$ the operation in the group $\cl_0(\disc(q))$. 
Then for any $[q'] \in \cl_0(\disc(q))$ 
one has: $[q' \star q'] \in \text{Cl}_\iy(q)$. 
In Section~\ref{elliptic} an application 
towards elliptic curves is demonstrated.

\bigskip

\section{Torsors of norm forms} \label{torsors}
From now on $R$ is a domain in which $2$ is invertible, with fraction field $k$. 
Schemes defined over $\Sp R$ are underlined,   
omitting the underline for the generic fiber over $k$. 
Given a binary primitive quadratic $R$-form 
$q:M \to R$, the \emph{orthogonal group} of $(M,q)$ 
is the affine $R$-group of its 
self isometries (\cite[p.8]{Knus}): 
$$ 
\Oq := \{ H \in \un{\mathrm{GL}}(M) : q \circ H = q \}. 
$$ 
As $2$ is a unit in $R$, the \emph{special orthogonal subgroup} $\SO$ is 
$\ker[\Oq \xrightarrow{\det} \un{\BG}_m]$ (\cite[\S 1,p.1]{Con1}), 
and since $\CO$ is an integral domain, $\det$ factors through the group 
$\un{\mu}_2 = \Sp R[t]/(t^2-1)$ (\cite[Lemme~4.3.0.21]{CF}), thus we may just write 
$\SO := \ker[\Oq \xrightarrowdbl{\det} \un{\mu}_2]$.     

\mk 

Let $A$ be a quadratic $R$-algebra. 
If $A$ has a basis $\Om$ over $R$, 
inducing a representation 
$\p_\Om:\un{\mathrm{Aut}}(A) \hookrightarrow \un{\mathrm{GL}}_2$,  
then the above norm $n_A$ coincides with $\det \circ \p_{\Om}$ (\cite[Def. 2.3]{Biesel}).  
In particular, let $\Om := \{1,\sqrt{\a}\}$ 
where $\a \neq 0$ is a nonsquare element of $R$.   
The quadratic $R$-algebra 
$A_\a := R \la \Om \ra = R \oplus \sqrt{\a} R$  
is closed under multiplication and contains $R$, 
thus carries a ring structure. 
The Weil restriction of scalars 
$\un{\mathrm{R}} := \mathrm{Res}_{A_\a/R}(\un{\BG}_m)$,  
is a two-dimensional $R$-group whose generic fiber is a $k$-torus.   
The group of points $\un{\mathrm{R}}(R)$, 
via its isomorphism with $A_\a^\times$ \cite[\S 7.6]{BLR}, 
naturally acts on $A_\a$ through its basis $\Om$, 
yielding a canonical embedding 
of $\un{\mathrm{R}}$ in $\un{\text{Aut}}(A_\a)$.  
We get a commutative diagram: 
$$
\xymatrix{                                
  \un{\mathrm{R}} \ar@{^{(}->}[r] \ar@{->>}[dr]^{n_\a} 
& \un{\mathrm{Aut}}(A_\a)  \ar@{^{(}->}[r]^-{\p_{\Om}} 
& \un{\mathrm{GL}}_2 \ar@{->>}[dl]_{\det}  \\ 
&     \un{\BG}_m  										
}.$$   

Consider the $R$-group 
$\un{\mathrm{N}}:=\ker[\un{\mathrm{R}} \xrightarrow{n_\a} \un{\BG}_m]$. 
Its generic fiber $\mathrm{N}$ is a one-dimensional $k$-torus.  
At any prime $\fp$ the map applied to the reductions 
$\ov{R}_\fp \xrightarrow{(n_\a)_\fp} (\ov{\BG}_m)_\fp$ 
cannot be trivial as $\ov{R}_\fp$ is two-dimensional, 
thus the local norm $(n_\a)_\fp$ is surjective, 
hence $n_\a$ as well.  
This surjectivity holds even at a ramified prime 
$\fp|\la \det(B_q)\ra$ in which  
$\un{\mathrm{R}}_\fp = \Sp \hCO_\fp[x,y,t]/(t(x^2-\a y^2) - 1)$ 
is not reductive. 

\begin{lem} \label{N is flat}
The scheme $\N$ is flat over $\Sp R$. 
\end{lem}

\begin{proof} 
The schemes in $\un{\mathrm{R}} \xrightarrow{n_\a} \un{\BG}_m$ 
are smooth (e.g., \cite[Cor. A.5.4]{CGP}), 
hence regular and Cohen-Macaulay, thus it suffices to check 
that all the geometric fibers of $\N=\ker(n_\a)$ are one-dimensional. 
This is guaranteed by the surjectivity of~$n_\a$.  
\end{proof}

The quadratic module $(A_\a,q_\a)$ where 
$$
\tilde{q}_\a(X,Y) = q_\a(X+\sqrt{\a}Y) := n_\a(X+\sqrt{\a}Y) = X^2-\a Y^2
$$ 
i.e., $q_\a=(1,0,-\a)$ is of type $A_\a$  
(\cite[III, p. 164]{Knus}). 
From now on, just $n$ and $q$ will stand for $n_\a$ 
and $q_\a$, respectively.

\begin{remark} \label{unique closed and flat subgroup}
As $\det(B_q)$ is squarefree and $2$ is invertible, 
according to \cite[Prop.~2.3.]{APS} 
$\Oq$ and $\SO$ are smooth (thus flat).  
Then by the correspondence between flat closed 
subschemes of $\Oq$ and closed subschemes 
of the generic fiber 
$\mathrm{O}_q$ \cite[Prop.~2.8.1]{EGAIV},   
$\SO$ is the unique flat and closed subgroup of $\Oq$ 
whose generic fiber is $\mathrm{SO}_q$. 
\end{remark}

\begin{lem} \label{N = O+}
$\SO = \N$. 
\end{lem}

\begin{proof} 
Recall that $\N \subset \un{\mathrm{SL}}(A_\a)$. Then: 
\begin{align*} 
\un{\mathrm{SO}}_{q} 
    &=         \{a \in \un{\mathrm{SL}}(A_\a):q\circ a=q\} \\ \nonumber 
    &\supseteq \{a \in \N:                 q\circ a=q\} \\ \nonumber     
    &= \{a \in \N:q(xa)=q(x)\cdot n(a)=q(x) \ \forall x \in A_\a \} \\ \nonumber  
    &= \{a \in \N: n(a)=1 \} = \N.     
\end{align*} 
Both groups $\SO$ and $\N$ are $\CO$-flat 
(see Remark \ref{unique closed and flat subgroup} 
and Lemma \ref{N is flat}) 
closed subgroups of $\un{\mathrm{O}}_{q}$. 
Thus, the injection holds for the generic fibers as well. 
This injection is an equality as both groups 
$\mathrm{SO}_q$ and $\mathrm{N}$ 
are one-dimensional $k$-tori, 
implying by Remark~\ref{unique closed and flat subgroup}, 
that $\un{\mathrm{SO}}_{q} = \N$. 
\end{proof}

Given an affine $R$-group scheme $\un{G}$, 
a \emph{$\un{G}$-torsor in the \'etale topology} 
is a sheaf of sets on $R$ 
equipped with a (right) $\un{G}$-action, 
which is locally trivial in the \'etale topology.  
The pointed set $H^1_\et(R,\un{G})$ classifies 
these $\un{G}$-torsors up to $R$-isomorphisms. 
The following correspondence is due 
to Giraud (see \cite[\S 2.2.4]{CF}): 

\begin{prop} \label{torsors correspondence}
Let $S$ be a scheme and $X_0$ an object of 
a fibered category of schemes defined over~$S$. 
Let $\mathrm{Aut}(X_0)$ be its $S$-group of automorphisms. 
Let $\mathfrak{Forms}(X_0)$ be the category of $S$-forms that 
are locally isomorphic for some topology to $X_0$,  
and let $\mathfrak{Tors}(\mathrm{Aut}(X_0))$ be the category of $\mathrm{Aut}(X_0)$-torsors in that topology. 
$$ \mathfrak{Forms}(X_0) \to \mathfrak{Tors}(\mathrm{Aut}(X_0)): \ X \mapsto \mathrm{Iso}(X_0,X) $$
is an equivalence of fibered categories. 
\end{prop}

In particular the category of torsors of 
$\Oq = \un{\mathrm{Aut}}(q)$ in the \'etale topology 
is equivalent to the one of $R$-schemes 
of the form $\mathrm{Iso}(q,q')$ 
where $q'$ is a quadratic $R$-form 
\'etale-equivalent to $q$.  
Its associated \emph{discriminant algebra} 
is $D(q') = (\wedge^2 M',\det(q'))$ 
and its isomorphism class in $H^1_\et(R,\un{\mu}_2)$ is its \emph{Arf invariant}.  

Being diagonal, $q=(1,0,-\a)$ admits the improper isometry $\diag(1,-1)$ defined over $R$, 
so $\Oq(R) \xrightarrow{\det} \un{\mu}_2(R)$ is surjective.  
Then \'etale cohomology applied to the exact sequence of smooth $R$-groups: 
\begin{equation} \label{det sequence}
1 \to \SO \to \Oq \xrightarrow{\det} \un{\mu}_2 \to 1,  
\end{equation}
gives rise to the exact sequence of pointed-sets: 
\begin{equation} \label{LES}
1 \to H^1_\et(R,\SO) \to H^1_\et(R,\Oq) \xrightarrow{\det_*} H^1_\et(R,\un{\mu}_2),   
\end{equation} 
where $\det_*([\text{Iso}(q,q')]) = [D(q')]-[D(q)]$ in $H^1_\et(R,\un{\mu}_2)$ 
(preserving the base point, \cite[IV, Prop. 4.3.4]{Gir}). 
By the exactness of \eqref{LES}, $\text{Iso}(q,q')$ 
represents a class in $H^1_\et(R,\SO)=\ker(\det_*)$ 
if and only if $q'$ is equivalent to $q$ 
in the \'etale topology and shares the Arf invariant of $q$.

\begin{lem} \label{H1 equals cl}
The map $\psi: H^1_\et(R,\SO) \cong \cl_0(\disc(q)): [\mathrm{Iso}(q,q')] \mapsto [q']$ 
is a bijection of pointed sets having a structure of an abelian group. 
\end{lem}

\begin{proof} 
A representative $\text{Iso}(q,q')$ in $H^1_\et(R,\SO)$ 
satisfies $[D(q')]=[D(q)]$ in $H^1_\et(R,\un{\mu}_2)$.  
The latter is equivalent to 
$\det(B_{q'}) = a^2\det(B_q)$ for some 
$a \in \BF^\times$ (\cite[III, \S 3.3]{Knus}), 
i.e., to $\disc(q')=\disc(q)$. 
Identifying these sets of representatives modulo 
proper isometries by $\psi$, $\cl_0(\disc(q))$ is identified 
with $H^1_\et(R,\SO)$ and inherits its structure of the abelian group, 
as $\SO=\N$ (Lemma \ref{N = O+}) is commutative.  
\end{proof}

\begin{example} \label{improper isomorphism} 
Let $q=(1,0,x)$ be a quadratic form defined over $R=\BF_5[x]$.  
Then $q'=(2,0,x)$ cannot represent a class in $H^1_\et(R,\SO)$ 
as $\disc(q) \neq \disc(q')$, though representing a class in $H^1_\et(R,\Oq)$; 
fixing $a$ such that $a^2=2$, $H=\diag(a,1)$ 
is an isometry $q \to q'$ defined over an \'etale cover of $R$. 
The form $q''=(4,0,x)$, however, being $R$-equivalent to $q$ by $H=\diag(2,1)$,  
represents a class in $H^1_\et(R,\SO)$, but not the one of $q$ as $H$ cannot be proper. 
\end{example}

Given a set $\Om'=\{\om_1,\om_2\}$, 
set the vector $\ov{\Om'} = 
\begin{pmatrix}
 \om_1 \\
 \om_2 \\ 
 \end{pmatrix}$ 
and notice that $q'=q \circ H$ can be written as  
$\tilde{q}'(X,Y) = \tilde{q}((X,Y) \cdot H)$, 
or, equivalently, as  
$$ 
q'((X,Y) \cdot \ov{\Om}) = q((X,Y) \cdot \ov{\Om'})  
$$
where $\ov{\Om'}=H \cdot \ov{\Om}$. 
We denote $\Om'$ briefly by $H \Om$.    

\begin{lem} \label{short LES}
The map 
$\wt{i}_*([q \circ H]) = [R\la H\Om\ra]$
forms an exact sequence of pointed-sets: 
\begin{equation*} 
1 \to R^\times / n(A_\a^\times) 
\to \cl_1(-\a (R^\times)^2) \xrightarrow{\wt{i}_*} 
\mathrm{Pic}(A_\a) \to 1.   
\end{equation*} 
\end{lem}

\begin{proof} 
Applying \'etale cohomology to the short exact sequence of smooth $R$-groups: 
\begin{equation} \label{SES} 
1 \to \N \xrightarrow{i} \un{\mathrm{R}} \xrightarrow{n} \un{\BG}_m \to 1,   
\end{equation}
being commutative, gives rise 
to a short exact sequence of abelian groups:
\begin{equation} \label{LES Gm} 
1 \to R^\times / n(A_\a^\times) \xrightarrow{\dl} H^1_\et(R,\N) \xrightarrow{i_*} H^1_\et(R,\un{\mathrm{R}}),  
\end{equation} 
in which $H^1_\et(R,\un{\mathrm{R}})$ 
is isomorphic as an abelian group by Shapiro's Lemma to 
$H^1_\et(A_\a,\un{\BG}_m)\cong \Pic(A_\a)$.  
A representative $\Iso(q,q \circ H)$ in $H^1_\et(R,\N=~\SO)$ 
(Lemma \ref{N = O+}) corresponds to the quadratic $R$-module 
$(R\la H\Om\ra,q \circ H)$ with trivial Arf-invariant. 
Its class is mapped by $i_*$ to $[\text{Iso}(A_\a,R \la H\Om \ra)]$  
in $H^1_\et(A_\a,\BG_m=\text{Aut}(A_\a))$ 
(\cite[V, 3.1.1.1]{Gir}), or, equivalently, 
to $[R\la H\Om\ra]$ in $\Pic(A_\a)$ 
(forgetting the quadratic form, 
see the proof of \cite[Thm. 3]{Kne}).  
By Lemma \ref{H1 equals cl} $i_*$ can be replaced by 
$i_*([q \circ H]) = [R\la H \Om \ra]$ in: 
\begin{equation}  
1 \to R^\times / n(A_\a^\times) 
\xrightarrow{\psi \circ \dl} 
\cl_0(-\a(R^\times)^2) \xrightarrow{i_*} \Pic(A_\a).   
\end{equation} 
The extension $\wt{i}_*$ of $i_*$ to $\cl_1(-\a(R^\times)^2)$, 
including all quadratic \emph{maps} of 
the discriminant $-\a(R^\times)^2$, is surjective;   
any representative $P$ in $\Pic(A_\a)$  
corresponds up to isomorphism to a primitive 
quadratic $R$-map $q_P$ such that the quadratic 
module $(P,q_P)$ is of type~$A_\a$ 
(see Section \ref{intro}).  
This $q_P$ which can be taken to be the norm $n_P$ 
induced by the standard involution on $P$ (\cite[III, Prop. 7.3.1]{Knus}),  
is locally everywhere equivalent to $n$ (\cite[III, Remark 7.3.4]{Knus}), 
implying that: 
\begin{equation} 
\exists a \in \bigcap\limits_{\fp \in \Sp R} \hCO_\fp^\times = R^\times: 
\det(B_{n_P}) = a^2 \det(B_n),  
\end{equation} 
thus $\disc(n_P)=\disc(n)$ and so $[n_P] \in \cl_1(-\a(R^\times)^2)$.  
This amounts in the asserted exact sequence. 
\end{proof}

\begin{remark} \label{abelian group structure}
Each fiber of the surjection $i_*:[(P,q_P)] \to [P]$, 
forgetting as above the quadratic map $q_P$, 
bijects with the kernel $R^\times / n(A_\a^\times)$,  
thus $\cl_1(-\a(R^\times)^2)$ can be endowed a-priori 
with the structure of an abelian group, 
being $R^\times / n(A_\a^\times) \times \Pic(A_\a)$.   
\end{remark}

\bk

\section{Dedekind correspondence: the geometric case} \label{Gauss section}
Let $\BF(x)$ be the field of rational functions over a finite field $\BF$ of odd characteristic.  
Any finite non-scalar extension $k$ of $\BF(x)$ can be viewed as the function field 
of some projective, smooth and geometrically connected $\BF$-curve $C$. 
Each closed point $\fp$ on $C$ gives rise to a discrete valuation on $k$. 
Let $\hCO_\fp$ be the ring of integers in the completion $\hat{k}_\fp$ of $k$ 
with respect to $v_\fp$. 
Throughout, $k$ is assumed \emph{imaginary}, namely,  
the prime $\iy = \la 1/x \ra$ in $\BF(x)$ 
does not split into distinct places in $k$ (cf. \cite{LM}).  
Let $\iy_k$ be the unique prime of $k$ lying 
above $\iy$, regarded as a closed point on $C$.  
Then the ring of regular functions on the affine curve 
$C^\af := C- \{\iy_k\}$ is a Dedekind domain:  
$$ 
\CO := \BF[C^\af] = \{x \in k: v_\fp(x) \geq 0 \ \forall \fp \neq \iy_k \}. 
$$

This leads to the geometric analogue of the Dedekind correspondence: 

\begin{theorem}  \label{Gauss} 
Given a finite field $\BF$ of odd characteristic, 
let $k/\BF(x)$ be an imaginary extension 
and $\CO$ the domain of $k$-elements regular
everywhere away of the prime $\iy_k$ lying over $\iy=\la 1/x \ra$.    
Let $K/k$ be a geometric imaginary quadratic extension 
and $\CO_K$ the domain of $K$-elements  
regular everywhere away of the prime $\iy_K|\iy_k$, 
with discriminant $\Dl_K = -\a (\BF^\times)^2$, $\a \in \CO$. 
There is an isomorphism of abelian groups:    
$$ \wt{i}_*:\cl_1'(\Dl_K) \xrightarrow{\sim} \mathrm{Pic}(\CO_K): 
\ \left[(a,b,c) \right] 
\mapsto \left[ \big \la a,b/2 + \sqrt{\a} \big \ra \right]. $$ 
\end{theorem}

\begin{proof} 
First we see that $\CO_K = A_\a$ for some (non-scalar) $\a \in \CO$.    
As a (maximal) order over $\CO$, $\CO_K$ is a free $\CO$-module of rank $2$,  
thus admits a basis $\{1,t\}$ over $\CO$,  
where $t$ is an algebraic integer thus a root of a monic quadratic polynomial over $\CO$,   
i.e., there exist $m,n \in \CO$ s.t. $t^2+mt+n=0$. 
Taking $t$ to be the root $\fc{\sqrt{m^2-4n}}{2}-\fc{m}{2}$, 
we see that $\CO_K = A_\a$ for $\a = m^2-4n$.   

\mk 

Consider the exact sequence in Lemma \ref{short LES} for $R=\CO$:
$$
1 \to \CO^\times / n(A_\a^\times) \xrightarrow{\dl} 
\cl_1(\Dl_K) \xrightarrow{\wt{i}_*} \mathrm{Pic}(A_\a) \to 1.   
$$ 
Notice that $\CO^\times = \BF^\times$, and as 
$K$ is imaginary one has 
$n(A_\a^\times)=(\BF^\times)^2$ 
(\cite[Example~1]{Mor}).  
A representative $\lm$ in $\CO^\times / n(A_\a^\times) \cong \BF^\times/(\BF^\times)^2$ 
is mapped by $\dl$ to $[\lm q] \in \cl_1(\Dl_K)$,   
hence $\cl_1(\Dl_K) / \Im(\dl) \cong \cl_1'(\Dl_K)$ 
(recall by Remark \ref{abelian group structure} 
that $\cl_1(\Dl_K)$ has a group structure). 

\mk 

Explicitly, starting by a general $\CO$-form 
$q_L = (a,b,c)$ of discriminant $\Dl_K=-\a (\BF^\times)^2$ 
($c$ can be taken to be $\fc{b^2/4-\a}{a}$), 
we have $q_L = q \circ H_L$ where 
$H_L = 
\fc{1}{\sqrt{a}} \left( \begin{array}{cc}
     a    &  0   \\ 
     b/2  & 1 \\ 
\end{array}\right)$. 
As $\det(H_L)=1$, $[q_L] \in \cl_1'(\Dl_K)$. 
By Lemma~\ref{short LES}:    
$$ \wt{i}_*([q_L]) = 
[\CO \la H_L \Om = \fc{1}{\sqrt{a}} \{a,b/2+\sqrt{\a}\} \ra] = 
[\la a,b/2+\sqrt{\a} \ra]$$ 
in $\Pic(\CO_K)$   
(the two ideals differ by tensoring with a principal one). 
\end{proof}

\begin{remark} \label{exact sequence over Z}
The exact sequence \eqref{LES Gm} was obtained 
similarly in \cite[(2.3)]{BS} for $R=\Z$, 
only for the flat cohomology, as $\N$ may not be smooth 
at $2$ which is not a unit in $\Z$, 
and the map $i_*$ is surjective there, 
as the following right term $\Pic(R)$ is trivial 
for $R=\Z$. In both cases $\ker(\dl)$ is $R^\times/n(A_\a^\times)$, 
but its non-trivial coset differs: while for $R=\CO$ it is represented by 
$\lm \in \BF^\times-(\BF^\times)^2$, for $R=\Z$ 
it is $-1$ (for $K=\Q(\sqrt{d}), d<0$). 
This explains in what sense the identification of two definite forms 
in our $\cl_1'(\Dl_K)$ in Def.~\ref{cl1},  
is analogous to the one over $\Z$. 
\end{remark}

\begin{remark} \label{opposite}
Given a form $q=(a,b,c)$, we call $q^\op = (a,-b,c)$ its \emph{opposite form}.   
By Theorem \ref{Gauss} the tensor product in $\Pic(\CO_K)$ 
induces by a group operation in $\cl_1'(\disc(q))$: 
$$ 
[q_{L_1}] \star [q_{L_2}] = \wt{i}_*^{-1}([L_1 \otimes L_2]),  
$$ 
and $[q]^{-1} = [q^\op]$. 
Indeed, let $L^\op$ be the ideal corresponding to $q^\op$. Then: 
\begin{align*} 
I = L \otimes L^\op &= \la a,\sqrt{a}+b/2 \ra \otimes \la a,\sqrt{\a}-b/2 \ra \\ \nonumber 
&= \la a^2,a(\sqrt{\a}+b/2),a(\sqrt{\a}-b/2),b^2/4-\a \ra. 
\end{align*} 
But $b^2/4-\a=ac$, thus $I \subseteq \la a \ra$. 
On the other hand, both $L$ and $L^\op$ are primitive,  
thus $\la a \ra \subseteq I$, whence $I = \la a \ra$ is principal. 
\end{remark}

\bk

\section{Not necessarily proper classification} \label{improper classification}
In this section we study a less narrow classification of all primitive 
binary quadratic $\CO$-maps of a common given discriminant, namely, 
up to \emph{proper and improper} $\CO$-isometries   
$$ \cl(\Dl) := \{ q:\disc(q) = \Dl\} / \mathrm{GL}_2(\CO). $$

Given a smooth $\CO$-group $\un{G}$ and a representative $P$ in $H^1_\et(\CO,\un{G})$, 
the quotient of $P \times_{\CO} \un{G}$ 
by the $\un{G}$-action, $(p,g) \mapsto (ps^{-1},sgs^{-1})$,  
is an affine $\CO$-group scheme ${^P}\un{G}$, 
being an inner form of $\un{G}$, 
called the \emph{twist} of $\un{G}$ by $P$ 
(e.g., \cite[\S 2.2]{Sko}). 

\begin{lem} \label{relation}
$\cl(\Dl) \cong \cl_1(\Dl)/([q] \sim [q^\op])$. 
\end{lem}

\begin{proof}
As $B_q=\diag(1,-\a)$, the short exact sequence of smooth $\CO$-groups \eqref{det sequence}:
\begin{equation*} 
1 \to \SO \to \Oq \xrightarrow{\det} \un{\mu}_2 \to 1
\end{equation*}
splits by the section mapping the non-trivial element in $\un{\mu}_2$ to $\diag(1,-1)$ in $\Oq$,  
i.e., $\Oq$ is isomorphic to $\SO \rtimes \un{\mu}_2$. 
Then according to \cite[Lemma~2.6.3]{Gil} we get:  
\begin{equation} \label{decomposition}
H^1_\et(\CO,\Oq) = \coprod_{[P] \in H^1_\et(\CO,\un{\mu}_2)} H^1_\et(\CO,^{P}\SO) / \un{\mu}_2(\CO), 
\end{equation}
in which $\un{\mu}_2(\CO)$ acts on the set of representatives of $H^1_\et(\CO,^{P}\SO)$ by $\diag(1,\pm 1)$. 
If $P$ is a trivial $\un{\mu}_2$-torsor, 
then $H^1_\et(\CO,^{P}\SO = \SO) \cong \cl_0(\Dl)$ by Lemma \ref{H1 equals cl}. 
Otherwise, if $P$ represents a non-trivial $\un{\mu}_2$-torsor, 
then the corresponding twisted form has a distinct Arf-invariant than the one of $q$, 
which thereby does not belong to $H^1_\et(\CO,\SO)$ 
(see sequence \eqref{LES}).  
Consequently, $\cl(\Dl)$ is identified with 
the first component in the decomposition \eqref{decomposition}, 
thus to $\cl_0(\Dl)$ extended to $\cl_1(\Dl)$ 
modulo $\un{\mu}_2(\CO)$, such that any quadratic map 
is identified with its opposite by $\diag(1,-1)$. 
\end{proof}

\mk

Applying the quotient in Lemma \ref{relation} 
to $\cl_1'(\Dl_K)$ gives by Theorem \ref{Gauss}:
 
\begin{cor}
Let $K/k$ be a geometric quadratic extension of imaginary fields.    
There is a bijection of pointed sets (compare with \cite[Remark~5.20]{BS}): 
$$ \cl'(\Dl_K) \cong \mathrm{Pic}(\CO_K) / (x \sim x^{-1}); 
\ \left[(a,b,c) \right] \mapsto \left[\big \la a,\sqrt{\a}-b/2 \big \ra \right] $$ 
thus $\cl'(\Dl_K)$ remains a group if and only if 
$\mathrm{exp}(\mathrm{Pic}(\CO_K)) \leq 2$.  
\end{cor}

\mk

\section{Genera and the principal genus} \label{class set}
Let $\un{G}$ be an affine $\CO$-group scheme with generic fiber $G$.  
The group of $k$-points $G(k)$ is embedded diagonally in the adelic group 
$\un{G}(\A)$, in which 
$\un{G}(\A_\iy):=G(k_\iy) \times \prod_{\fp \neq \iy_k} \un{G}(\CO_\fp)$ 
is also a subgroup. 
The \emph{class set} of $\un{G}$ 
is the finite set of double cosets (\cite[Prop.~3.9]{BP}):   
$$
\mathrm{Cl}_\iy(\un{G}) := \un{G}(\A_\iy) \backslash \un{G}(\A) / G(k).
$$    

Given furthermore, that $\un{G}$ is of finite type 
and smooth (not necessarily connected), 
it suits by Y.~Nisnevich \cite[Theorem~I.3.5]{Nis} 
into an exact sequence of pointed sets
\begin{equation} \label{Nis sequence}
1 \to \text{Cl}_\iy(\un{G}) \to H^1_\et(\CO,\un{G}) \xrightarrow{\vp} H^1(k,G) \times \prod_{\fp \neq \iy_k} \ H^1_\et(\hat{\CO}_\fp,\un{G}_\fp).    
\end{equation}

Let $[\xi_0] := \vp([\un{G}])$. 
The \emph{principal genus} of $\un{G}$ is then $\vp^{-1}([\xi_0])$, 
i.e., the set of classes of $\un{G}$-torsors that are generically and locally trivial at all primes of $\CO$. 
More generally, a \emph{genus} of $\un{G}$ is any fiber $\vp^{-1}([\xi])$ where $[\xi] \in \Im(\vp)$.  
The \emph{set of genera} of $\un{G}$ is then: 
$$ \text{gen}(\un{G}) := \{ \vp^{-1}([\xi]) \ : \ [\xi] \in \Im(\vp) \}, $$ 
whence $H^1_\et(\CO,\un{G})$ is a disjoint union of its genera. 
The left exactness of sequence \eqref{Nis sequence} reflects the fact that $\text{Cl}_\iy(\un{G})$ 
coincides with the principal genus of $\un{G}$.  
If there is an embedding   
\begin{equation} \label{local embedding} 
\forall \fp \neq \iy_k: 
H^1_\et(\hat{\CO}_\fp,\un{G}_\fp) \hookrightarrow H^1(\hat{k}_\fp,G_\fp) 
\end{equation} 
then as in~\cite[Cor.~I.3.6]{Nis}, the sequence~\eqref{Nis sequence} simplifies to
\begin{equation} \label{Nis short}
1 \to \text{Cl}_\iy(\un{G}) \to H^1_\et(\CO,\un{G}) \to H^1(k,G),  
\end{equation}
which indicates that any $\un{G}$-torsor 
belongs to the principal genus of $\un{G}$ if and only if it is $k$-isomorphic to it. 
More precisely, there is an exact sequence of pointed sets (cf.~\cite[Cor. A.8]{GP}) 
\begin{equation} \label{Nis short right exact}  
1 \to \mathrm{Cl}_\iy(\un{G}) \to H^1_\et(\CO,\un{G}) \to B \to 1,   
\end{equation}
in which 
$$ 
B = \left\{ [\g] \in H^1(k,G):\forall \fp \neq \iy_k,[\g \otimes \hat{\CO}_\fp] \in \Im \left(H^1_\et(\hat{\CO}_\fp,\un{G}_\fp) \to H^1(\hat{k}_\fp,G_\fp) \right) \right\}. 
$$

Let $K/k$ be a finite Galois extension, $\fp$ be a prime of $k$ and $\fP$ 
be a prime of $K$ dividing $\fp$. 
Write $\hat{k}_\fp$ and $\hat{K}_\fP$ for the completions of $k$ at $\fp$ 
and of $K$ at $\fP$, respectively, noting that $\hat{K}_\fP$ is independent 
of the choice of $\fP$ up to isomorphism. 
The norm map $\Nr : K \to k$ extends the above norm 
$n: \CO_K \to \CO$ and induces local maps $\mathrm{Nr}: K \otimes_k \hat{k}_\fp \to \hat{k}_\fp$; 
under the isomorphism above this corresponds to the product 
of the norm maps $\mathrm{Nr}_{K_\fP / k_\fp}$ on the components. 
Similarly, $\CO_K \otimes_{\CO} \hCO_\fp \simeq \CO_{\hat{K}_\fP}^r$. 
Write $U_\fp$ and $U_\fP$ for $\hat{\CO}_{\fp}^\times$ and $\CO_{\hat{K}_\fP}^\times$, respectively. 
The short exact sequence of smooth $\hCO_\fp$-groups (see Section \ref{torsors})
$$ 1 \to \N_\fp \to \un{\mathrm{R}}_\fp \to (\un{\BG}_m)_\fp \to 1 $$
yields by \'etale cohomology the exact and functorial sequence   
$$ 
1 \to \N_\fp(\hat{\CO}_\fp) \to \un{\mathrm{R}}_\fp(\hCO_\fp) \cong U_\fP^r \xrightarrow{\mathrm{Nr}} U_\fp \to H^1_\et(\hCO_\fp,\N_\fp) \to 1, 
$$
since $H^1_\et(\hCO_\fp, \un{\mathrm{R}}_\fp)$ 
is the Picard group of a product of local rings and thus vanishes.  

\mk 

We deduce an isomorphism $H^1_\et(\hCO_\fp,\N_\fp) \cong U_\fp / \mathrm{Nr}(U_\fP^r) = U_\fp / \mathrm{Nr}_{\hat{K}_\fP / \hat{k}_\fp}(U_\fP)$. 
Applying Galois cohomology to the short exact sequence of $\hat{k}_\fp$-groups
$$ 1 \to \mathrm{N}_\fp \to \mathrm{R}_\fp \to (\BG_m)_\fp \to 1 $$
gives rise to the exact sequence of abelian groups
$$ 1 \to \mathrm{N}_\fp(\hat{k}_\fp) \to \un{\mathrm{R}}_\fp(\hat{k}_\fp) \cong (\hat{K}_\fP^\times)^r \xrightarrow{\mathrm{Nr}} \hat{k}_\fp^\times \to H^1(\hat{k}_\fp,\mathrm{N}_\fp) \to 1, $$
where the rightmost term vanishes by Hilbert's Theorem~90.
Hence we may again deduce a functorial isomorphism $H^1(\hat{k}_\fp,\mathrm{N}_\fp) \cong \hat{k}_\fp^\times / \mathrm{Nr}_{\hat{K}_\fP / \hat{k}_\fp}(\hat{K}_\fP^\times)$. 
Note that $U_\fP$ is compact and thus $\mathrm{Nr}_{\hat{K}_\fP / k_\fp}(U_\fP)$ is closed in $k_\fp^\times$. 
Only units have norms that are units, so we obtain an embedding of groups: 
\begin{equation} \label{embedding of local N}
H^1_\et(\CO_\fp,\N_\fp) \cong U_\fp / \mathrm{Nr}_{\hat{K}_\fP / \hat{k}_\fp}(U_\fP) \hookrightarrow \hat{k}_\fp^\times / \mathrm{Nr}_{\hat{K}_\fP / \hat{k}_\fp}(\hat{K}_\fP^\times) \cong H^1(\hat{k}_\fp,\mathrm{N}_\fp). 
\end{equation}

\newpage 

\begin{definition} \label{Sha}
Let $S$ be a non-empty finite set of primes of $k$. 
The \emph{first Tate-Shafarevich set} 
of $G$ over $k$ relative to $S$ is 
$$ \Sh^1_S(k,G) := \ker\left[H^1(k,G) \to \prod_{\fp \notin S} H^1(\hat{k}_\fp,G_\fp) \right]. $$
\end{definition}

\begin{prop} \label{h(N)}
Suppose $[K:k]$ is prime and 
$\Nr(\un{\mathrm{R}}(\hCO_\fp)) = U_\fp \cap \Nr_{\hat{K}_\fP / \hat{k}_\fp}(\hat{K}_\fP^\times)$ for all $\fp$. 
Let $S_r$ be the set of primes dividing $\Delta_k$. 
Then there is an exact sequence of abelian groups (compare with formula (5.3) in \cite{Mor}):  
$$ 
1 \to \mathrm{Cl}_\iy(\N) \to H^1_\et(\CO,\N) 
\to \Sh^1_{S_r \cup \{\iy_k\}}(k,\mathrm{N}) \to 1. 
$$ 
\end{prop}

\begin{proof} 
As $H^1_\et(\CO_\fp,\N_\fp)$ embeds into $H^1(\hat{k}_\fp,\mathrm{N}_\fp)$ for any prime $\fp$ by~\eqref{embedding of local N}, 
the group $\N$ admits the exact sequence~\eqref{Nis short right exact}, 
consisting of abelian groups as $\N$ is commutative.  
The pointed set $\text{Cl}_\iy(\N)$ is in bijection with the first Nisnevich cohomology set 
$H^1_{\text{Nis}}(\CO,\N)$ 
(cf.~\cite[I.~Theorem~2.8]{Nis}), which is a subgroup of $H^1_\et(\CO,\N)$ 
because any Nisnevich cover is flat. Hence the first map is an embedding. 
Since $K/k$ has prime degree and so is necessarily abelian, 
at any prime $\fp$ the local Artin reciprocity law implies that
$$ n_\fp=|\text{Gal}(\hat{K}_\fP/\hat{k}_\fp)|=[\hat{k}_\fp^\times:\mathrm{Nr}_{\hat{K}_\fP/\hat{k}_\fp}(\hat{K}_\fP^\times)]=|H^1(\hat{k}_\fp,\mathrm{N}_\fp)|. $$
Furthermore, since $[K:k]$ is a prime number, 
any ramified place $\fp$ is totally ramified,   
which implies that $[U_\fp : U_\fp \cap \mathrm{Nr}_{\hat{K}_\fP / \hat{k}_\fp}(U_\fP)] = n_\fp$ \cite[Theorem~5.5]{Haz}. 
Together with~\eqref{embedding of local N} this means that $H^1_\et(\hat{\CO}_\fp,\N_\fp)$ 
coincides with $H^1(\Q_\fp,\mathrm{N}_\fp)$ 
at ramified primes and vanishes elsewhere. 
Thus the set $B$ of~\eqref{Nis short right exact} consists of classes $[\g] \in H^1(k,\mathrm{N})$ 
whose fibers vanish at unramified places.   
This means that $B=\Sh^1_{S_r \cup \{\iy_k\}}(k,\mathrm{N})$, where $S_r$ is the (finite)  
set of ramified primes of $K/k$.
\end{proof}

\begin{remark} \label{B}
The group $B=\Sh^1_{S_r \cup \{\iy_k\}}(k,\mathrm{N})$ 
embeds in $H^1(k,\mathrm{N})$ by definition. 
But $H^1(k,\mathrm{N}) \cong k^\times/ \mathrm{Nr}(K^\times)$, which means that $B$ has an 
exponent dividing $n=[K:k]$.  
\end{remark}

\begin{cor} \label{exponent 2}
For any $[q'] \in \cl_0(\disc(q))$, 
$[q' \star q'] \in \text{Cl}_\iy(q):= \text{Cl}_\iy(\SO)$. 
\end{cor}

\begin{proof}
The scheme $\SO=\N$ admits the 
exact sequence~\eqref{Nis short right exact} 
(see the proof of Prop.~\ref{h(N)}).  
Together with \eqref{Nis short} and Lemma \ref{N = O+} the principal genus satisfies 
\begin{equation*} 
\text{Cl}_\iy(\SO=\N) = \ker[H^1_\et(\CO,\N) \to H^1(k,\mathrm{N})].  
\end{equation*} 
The quotient $H^1_\et(\CO,\N) / \text{Cl}_\iy(\N) = \Sh^1_{S_r \cup \{\iy_k\}}(k,N)$ 
has exponent $2$ by Proposition~\ref{h(N)} and Remark~\ref{B}. 
This means that if $[q'] \in H^1_\et(\CO,\N=\SO)$, 
the latter pointed-set being bijective to $\cl_0(\disc(q))$ 
by Lemma \ref{H1 equals cl}, then $[q' \star q']$ lies in $\text{Cl}_\iy(q)$. 
\end{proof}

\bk

\section{Over elliptic curves} \label{elliptic}
Let $\CO=\BF[x]$ and so $k=\BF(x)$. 
Let $C = \{ Y^2Z = X^3 + aXZ^2 + bZ^3\}$ be an elliptic curve defined over $\BF$.  
Then $K=\BF(C)$ is quadratic imaginary over $k$;  
$K=k(\sqrt{\a})$ where $\a=x^3+ax+b \in \CO$.    
As $\text{char}(k)$ is odd $K/k$ is separable  
and as $\deg(\a)=3$, $\iy_k=\la 1/x \ra$ ramifies in $K$ (\cite[Theorem~1(1)(a)]{DLB}). 
Suppose $\iy_K=(0:1:0)$ belongs to $C(\BF)$. 
Then $\CO_K=\BF[C^\af]$ where $C^\af$ is the affine $\BF$-curve 
$C-\{\iy_K\}=\{y^2=\a \}$ and one has 
$\Pic(\CO_K) \cong C(\BF)$ (e.g., \cite[Example~4.8]{Bit}). 
Let as above $\cl_1'(\Dl_K)$ be the set of classes 
of primitive quadratic binary $\CO$-forms, 
being definite, 
up to proper $\CO$-isometries. 
By Theorem \ref{Gauss} one has: $\cl_1'(\Dl_K) \cong C(\BF)$. 
Explicitly,   

\begin{cor} \label{C via cl}
Let $C = \{ Y^2Z = X^3 + aXZ^2 + bZ^3\}$ 
be an elliptic $\BF$-curve 
such that $\iy_K := (0:1:0) \in C(\BF)$. 
Set: $\Dl_K = -(x^3+ax+b) (\BF^\times)^2$.  
Then there is an isomorphism of abelian groups $C(\BF) \cong \cl_1'(\Dl_K)$ given by: 
\begin{align*}
[(\mA:\B:\C \neq 0)] 
\mapsto \left[\left(x-\fc{\mA}{\C},-\fc{2\B}{\C},\fc{\left(\fc{\B}{\C}\right)^2-\a}{x-\fc{\mA}{\C}}\right)\right],  
[(0:1:0)] \mapsto [(1,0,-\a)].     
\end{align*} 
\end{cor} 

\begin{proof}
Since $\iy_K$ is a closed point on $C$, 
$\BF[C-\{\iy_K\}] = \CO[\sqrt{\a}]$ where $\a=x^3+ax+b$ 
is the ring of $\{\iy_K\}$-integers in $K=\BF(C)$. 
The above correspondence is then given by: 
\begin{align*} 
[(\mA:\B:\C \neq 0)] \in C(\BF) - \{\iy_K\} 
&\mapsto \left[(\mA/\C,\B/\C) \right] \in C^\af(\BF) \\ \nonumber  
&\mapsto \left[\big \la x-\mA/\C,y-\B/\C \big \ra \right] \in \Pic(\CO_K)  \\ \nonumber 
& \mapsto \left[\left(x-\mA/\C,-2\B/\C,\fc{\left(\B/\C\right)^2-\a}{x-\mA/\C}\right) \right] \in \cl_1'(\Dl_K)   
\end{align*}
and $(0:1:0) \mapsto (1,0,-\a)$. 
\end{proof}

\begin{example}
Let $C = \{ Y^2Z = X^3 + XZ^2 + Z^3 \}$ defined over $\BF_3$. 
Removing the rational point $\iy_K = (0:1:0)$, 
we get the affine curve 
$C^\af = \{ y^2 = x^3 + x + 1 \}$ with 
$\CO_K = \BF_3[x,y]/\la y^2 - x^3 - x - 1 \ra$.  \\
Then we have: 
{
\begin{center}
\begin{tabular}{|c | c | c | c| c| c| c|} 
 \hline
  i  &  $C(\BF_3)$  & affine support  & order & $L_i$       &  $q_{i}$            \\ [0.5ex] \hline\hline
  1  &  $(1:0:1)$   & $(1,0)$          & 2     & $(x-1,y)$   &  $(x-1,0,2x^2+2x+1)$  \\ \hline
  2  &  $(0:1:2)$   & $(0,2)$          & 4     & $(x,y-2)$   &  $(x,2,2x^2+2)$       \\ \hline
  3  &  $(0:1:1)$   & $(0,1)$          & 4     & $(x,y-1)$   &  $(x,1,2x^2+2)$       \\ \hline
  4  &  $(0:1:0)$   & $\text{O}$       & 1     & $\CO_K$     &  $(1,0,2x^3+2x+2)$    \\ \hline
\end{tabular}
\end{center}} 
Here $q_2$ and $q_3$ are opposite  
$[q_2] \star [q_3] = [q_4]$. 
Indeed: $\la x,y-2 \ra \otimes \la x,y-1 \ra = \la x \ra$, thus 
$$\cl_1'(\Dl_K) \cong \Pic(\CO_K) \cong \Z/4, $$ 
while $\cl'(\Dl_K) = \{ [q_1],[q_2],[q_4] \}$ 
has no group structure. 

According to Proposition \ref{h(N)} 
there are $2$ genera   
and by Proposition~\ref{exponent 2} the class 
$[q_2^2]$ belongs to the principal genus though not being the trivial one. 
Indeed: $[q_2^2]=[q_1]$ by the group law,  
and as $y^2=(x-1)(x^2+x-1)$, $q_1$ is isomorphic to $q_4$ by 
$\left( 
     1/\sqrt{x-1},  
     \sqrt{x-1} \right)$  
locally at $\fp \neq \la x-1 \ra$, and by 
$\diag\left( \sqrt{x^2+x-1}/y, y/\sqrt{x^2+x-1} \right)$ 
at $\fp = \la x-1 \ra$.  
\end{example}

\begin{example}
Let $C = \{Y^2Z = X^3 + XZ^2 \}$ defined over $\BF_5$.  
Removing the rational point $\iy_K = (0:1:0)$ 
we get the affine curve 
$C^\af = \{y^2 = x^3 + x \}$ with 
$\CO_K = \BF_5[x,y]/ \la y^2-x^3-x \ra$. 
Then: 
\mk 
\begin{center}
 \begin{tabular}{|c | c | c | c| c| c| c|} 
 \hline
 i  &  $C(\BF_5)$  & affine support   &  order  &  $L_i$      &  $q_{i}$         \\ [0.5ex] \hline\hline
 1  &  $(0:0:1)$   & $(0,0)$          &  2      &  $(x,y)$    &  $(x,0,4x^2+4)$    \\ \hline
 2  &  $(1:0:2)$   & $(3,0)$          &  2      &  $(x-3,y)$  &  $(x-3,0,x^2+3x)$  \\ \hline
 3  &  $(1:0:3)$   & $(2,0)$          &  2      &  $(x-2,y)$  &  $(x-2,0,x^2+2x)$  \\ \hline
 4  &  $(0:1:0)$   & $\text{O}$       &  1      &  $\CO_K$    &  $(1,0,4x^3+4x)$    \\ \hline
\end{tabular}
\end{center}
\mk 
Here we observe no forms of order greater than $2$ and so 
$$\cl'(\Dl_K) = \cl_1'(\Dl_K) \cong \Pic(\CO_K) \cong \Z_2^2.$$ 
\end{example}

\mk 

{\bf Acknowledgements:} 
I thank U.~First, P. Gille, B.~Kunyavski\u\i, S.~Scully and 
S. Vladuts for valuable discussions concerning the topics of the present article.  
I also thank the anonymous referee for his 
constructive remarks. 

\mk

\mk 

\end{document}